\documentclass[12pt]{article}

\usepackage[top=0.75in]{geometry}
\usepackage{amsmath,amsthm,amssymb}
\usepackage{enumerate}

\newtheorem{thm}{Theorem}

\newtheorem{prop}[thm]{Proposition}
\newtheorem{cor}[thm]{Corollary}

\theoremstyle{definition}
\newtheorem*{defi}{Definition}
\newtheorem*{rem}{Remark}
\newtheorem*{conj}{Conjecture}

\begin{document}

\title{Game of Pure Chance with Restricted Boundary}
\author{Ho-Hon Leung \\
Department of Mathematical Sciences \\
United Arab Emirates University \\
Al Ain, 15551\\
United Arab Emirates\\
{\tt hohon.leung@uaeu.ac.ae}\\
\ \\ Thotsaporn ``Aek'' Thanatipanonda\\
Science Division\\
Mahidol University International College\\
Nakornpathom, Thailand \\
{\tt thotsaporn@gmail.com}}

%\date{October 1, 2019}
\date{January 14, 2020}

\maketitle
\thispagestyle{empty}

%%%%%%%%%%%%%%%%%%%%%%%%%%%%%%%%%%%
\begin{abstract}
We consider various probabilistic games with piles for one player or two players. In each round of the game, a player randomly chooses to add $a$ or $b$ chips to his pile under the condition that $a$ and $b$ are not necessarily positive. If a player has a negative number of chips after making his play, then the number of chips he collects will stay at $0$ and the game will continue. All the games we considered satisfy these rules. The game ends when one collects $n$ chips for the first time. Each player is allowed to start with $s$ chips where $s\geq 0$. We consider various cases of $(a,b)$ including the pairs $(1,-1)$ and $(2,-1)$ in particular. We investigate the probability generating functions of the number of turns required to end the games. We derive interesting recurrence relations for the sequences of such functions in $n$ and write these generating functions as rational functions. As an application, we derive other statistics for the games which include the average number of turns required to end the game and other higher moments. 
\end{abstract}

This article is accompanied by Maple package: {\tt Pile2.txt}. 
They are available on the website\\

{\tt http://www.thotsaporn.com/GPCBoundary.html} .

%%%%%%%%%%%%%%%%%%%%%%%%%%%%%%%%%%%
\section{Introduction} \label{section1}
A {\it game of pure chance} is a probabilistic game in which no skill is involved. Each participant takes turns to play the game randomly. He is not required to spend hours of practice to improve his skills in the game. Hence, a beginner in the game is not in a disadvantageous position when he plays against a professional player. These games are relaxing and could be exciting sometimes. We always play such games in our lives. Lottery, Roulette, ``Snakes and Ladders'' and Dreidel are some examples of games of pure chance.

In this paper, we analyze games of pure chance \textbf{with restricted boundary}, namely \textbf{pile games with boundary}. They can be viewed as solitaire games or as competition among multiple players. Mathematically, we describe the game as follows: Given a fixed number $n$ and a {\it set of choices} $R=\{r_1,r_2,...,r_m\}$ such that each element $r_i\in R$ is an integer  and is assigned a transition probability, a player in the game randomly chooses an element $r_i\in R$ in each round; and he takes $r_i$ chips accordingly. The game ends when one of the players collects $n$ chips or more. We note that the integers $r_i\in R$ are not necessarily positive. If a player has a negative number of chips after making his play, then the number of chips he collects will stay at $0$ and the game will continue. As an example about the set of choices $R$, in Roulette, we have $R=\{-1,35\}$ such that $p(-1)=36/37$ and $p(35)=1/37$. 

For a pile game with boundary, the only absorbing state is reached when a player collects $n$ chips or more. We compare this game to the famous \textit{Gambler's Ruin} problem in which a player is allowed to have a negative number of chips; and the game ends when a player collects some positive or negative number of chips. The pile games with boundary is analogous to the real world situation when someone goes bankrupt. He does not have to pay the debt and is allowed to start fresh after the end of the bankruptcy period.

The paper is organized as follows: In Section \ref{One} and Section \ref{Two}, we study the pile games with boundary for the set of choices $R=\{1,-1\}$ in the single-player scenario and the two-player scenario respectively. We derive properties for the generating functions of the probabilities to end the game at different turns/rounds. Along the way we compute other statistics, say, the average number of turns required to end the game and higher moments of such random variables. We also discover interesting patterns obtained from sequences of numbers that are naturally associated with the games. In Section \ref{section4}, we do a similar study on the pile game with boundary for the general case of $R=\{1,-u\}$ where $u\geq 1$. In Section \ref{section5}, we do the same investigation about the pile game with boundary for the set of choices $R=\{2,-1\}$. We single out this case because the patterns involved in the generating functions get quite complicated and are rather different from the cases considered in previous sections. In the last section, we make the final conclusion on the main results we obtained in this paper.

%-------------------------------------------------------------------------
\subsection{Single-player scenario} \label{section1.1}

\begin{defi}
Let $a_n(k)$ be the number of ways to end the game, i.e., the player collects $n$ chips or more for the first time at the $k^{\text{th}}$ turn. 
\end{defi}

\begin{defi}
Let $b_n(k,s)$ be the number of ways to end the game, i.e., the player collects $n$ chips or more for the first time at the $k^{\text{th}}$ turn while starting with $s$ chips.
\end{defi}

We note that $b_n(k,0)=a_n(k)$. It is clear that the quantities 
$b_{n}(k,s)$ satisfy a system of linear equations. 
For example, if we have a set of choices $R=\{1,-1\}$, then for $n > 0$ and $k\geq 1$, we have
\begin{align} 
b_n(k,0)  &= b_n(k-1,0)+b_n(k-1,1) ,  \label{b1} \\
b_n(k,s)  &= b_n(k-1,s-1)+b_n(k-1,s+1) ,   \;\  1 \leq s <n, \label{b2}  \\
b_n(k,n) &= 0, \label{b3}
\end{align}
while $b_n(0,n)=1$ and $b_n(0,s)=0, 0\leq s<n$. 

\begin{defi}
Let $B_n(k,s)$ be the probability
that the game ends at the $k^{\text{th}}$ turn while starting with $s$ chips.
The probability generating function $G_{n,s}(x)$ is defined by
\[G_{n,s}(x) := \sum_{k=0}^{\infty} B_{n}(k,s)x^k.\]
\end{defi}

For example, if we have a set of choices $R=\{1,-1\}$ such that the probability to add one chip to the pile is $p$ and the probability to remove one chip from the pile is $q=1-p$, then for $n > 0$, we have
\begin{align} 
G_{n,0}(x)  &= pxG_{n,1}(x) +qxG_{n,0}(x),  \label{Sys1:a} \\
G_{n,s}(x)  &= pxG_{n,s+1}(x) +qxG_{n,s-1}(x),   \;\  1\leq s <n, \label{Sys1:b} \\
G_{n,n}(x) &= 1. \label{Sys1:c}
\end{align}

It is important to note that, by Cramer's rule, 
for any fixed number $n$, the function $G_{n,s}(x)$
% that arises from this system of linear equations 
is a rational function in $x$. 
Also, the denominators of $G_{n,s}(x)$
are the same for all $s$ and have degrees of at most $n$.

We state our observations for $B_n(k,s)$ and $G_{n,s}(x)$
as a theorem below in Theorem \ref{Concept1}. This is a \textit{conceptual theory} from 
Theorem 1 of the new bible of games of pure chance \cite{TZ}. Before stating the theorem, we give the definition of the {\it C-finite sequence} to readers who are unfamiliar with this subject. The textbook written by Kauers and Paule \cite{KP} and the paper written by the second author and Zeilberger \cite{Z1} serve as excellent references to readers who want to learn more about this topic. 

\begin{defi}
A sequence $a(n)$ is called a C-finite sequence if it fits the $C$-finite ansatz, i.e., the terms 
$a(n)$ satisfy a linear recurrence relation with constant coefficients as follows:
\begin{align}  
a(n)+c_1a(n-1)+...+c_La(n-L)&=0  \label{C} 
\end{align}
for some positive integer $L$ and constants $c_1,c_2,\dots,c_L$.
Equivalently, let $f(x)$ be the generating function for the sequence $a(n)$, i.e.,
\begin{align} f(x) &= \sum_{n=0}^{\infty} a(n)x^n. \nonumber\end{align}
Then, the function $f(x)$ is a rational function in $x$, i.e., the function $f(x)$ can be written as $P(x)/Q(x)$.
\end{defi}

It is clear from the definition of generating functions (see also \cite[Theorem 4.3]{KP})
that by $\eqref{C}$, the denominator $Q(x)$ of the rational function $f(x)$ can be written as follows:
\begin{align} 
Q(x) &= 1+c_1x+\dots+c_Lx^L. \label{Q}  
\end{align}

These C-finite sequences are the main objects of study in this paper. Now, we state our observations for $B_n(k,s)$ and $G_{n,s}(x)$ as follows:
  
\begin{thm}  \label{Concept1}
For any $n$, the probability generating function $G_{n,s}(x),
0 \leq s \leq n,$
are rational functions with the same denominator 
of degree at most $n$. Equivalently, the numbers $B_n(k,s)$, considered as a sequence in $k$ for fixed $n$ and $s$,  
is a C-finite sequence; and the recurrence relations are the same
for each value of $s$ (with different initial values, of course).
\end{thm} 

We will derive the statistics of $B_{n}(k,s)$ 
from its generating function once we have one.
For example, the formula for the average number of turns is
\[E[X]=x\dfrac{d\, G_{n,s}(x)}{dx} \bigg\rvert_{x=1};\]while the formula for the $r^{th}$ moment is 
\[\left(x\dfrac{d\,}{dx}\right)^r G_{n,s}(x) \bigg\rvert_{x=1}.\]

%---------------------------------------------------------------------------------------
\subsection{Two-player Scenario} \label{section1.2}
For the two-player scenario, both players 
compete to collect $n$ chips first. The game ends when one of the players collect $n$ chips for the first time.

\begin{defi}
Let $w_{n}(k,s_1,s_2)$ be the probability that the first player wins the game (by collecting $n$ chips or more first) 
at the $k^{\text{th}}$ turn when the first player and the second player start with $s_1$ and $s_2$ 
chips respectively. Similarly, we define
the losing probability $l_{n}(k, s_1,s_2)$
to be the probability that the first player loses the game (the second player
being the first to collect $n$ chips or more) at the $k^{\text{th}}$ turn. 

The winning and losing probability generating functions are defined by
\[W_{n,s_1,s_2}(x) := \sum_{k=0}^{\infty} w_{n}(k,s_1,s_2)x^k \quad \text{and} \quad L_{n,s_1,s_2}(x) := \sum_{k=0}^{\infty} l_{n}(k,s_1,s_2)x^k\]respectively.
\end{defi}

\begin{defi}
Let $t_{n}(k,s_1,s_2)$ be the probability
that the game ends (one of the players collects $n$ chips or more) 
at the $k^{\text{th}}$ round (one round consists of two turns,
one from each player) 
when the first player and the second player start with $s_1$ and $s_2$ 
chips respectively. 
\end{defi}

The generating function of the terms $t_{n}(k,s_1,s_2)$ is defined by
\[T_{n,s_1,s_2}(x) := \sum_{k=0}^{\infty} t_{n}(k,s_1,s_2)x^k.\]
It is clear that the relation between the functions $W_{n,s_1,s_2}(x)$, $L_{n,s_1,s_2}(x)$ and $T_{n,s_1,s_2}(x)$ is
\begin{equation} \label{MakeT}
  T_{n,s_1,s_2}(x) =\dfrac{1}{x}\cdot W_{n,s_1,s_2}(x^2)+L_{n,s_1,s_2}(x^2).   
  \end{equation}

For example, if we have the set of choices $R=\{1,-1\}$ such that the probability to add one chip is $p$ and the probability to remove one chip is $q=1-p$, then for $n > 0$, we have
\begin{align} \label{Sys2}
T_{n,0,s_2}(x)  &= pxT_{n,s_2,1}(x) +qxT_{n,s_2,0}(x),   \;\ 0 \leq s_2 < n, \\
T_{n,s_1,s_2}(x)  &= pxT_{n,s_2,s_1+1}(x) +qxT_{n,s_2,s_1-1}(x),   
\;\  1 \leq s_1 < n, \;\   0 \leq s_2 < n,\\
T_{n,n,s_2}(x) &= 1,    \;\   0 \leq s_2 < n, \\
T_{n,s_1,n}(x) &= 1,    \;\   0 \leq s_1 < n.
\end{align}

We make an observation similar to those in Section \ref{section1.1} here. By Cramer's rule,
the function $T_{n,s_1, s_2}(x)$ must be a rational function. 
Similar systems of equations apply to $W_{n,s_1, s_2}(x)$  and  $L_{n,s_1, s_2}(x)$ as well.
For example, for $n > 0$, we have
\begin{align*}  \label{SysW} 
 W_{n,s_1,s_2}(x)  &= x(p^2W_{n,s_1+1,s_2+1}(x) +pqW_{n,s_1+1,(s_2-1)^+}(x) \\
&+qpW_{n,(s_1-1)^+,s_2+1}(x) +q^2W_{n,(s_1-1)^+,(s_2-1)^+}(x),   
\;\  0 \leq s_1,s_2 < n, \\
W_{n,n,s_2}(x) &= 1,    \;\   0 \leq s_2 \leq n, \\
W_{n,s_1,n}(x) &= 0,    \;\   0 \leq s_1 < n,
\end{align*}
where $z^+ = \max{(0,z)}$.
The winning probability of the first player 
at the starting position is simply the following expression: 
\[   W_{n,0, 0}(x) \bigg\rvert_{x=1}. \]

The multi-player scenario can be explained similarly. 
This leads us to the second conceptual theory based on our previous results in \cite{TZ}.

\begin{thm} \label{Concept2}
In a pile game with boundary, the winning/losing probability generating function for
two or more player is a rational function.
Hence, the winning/losing probability is a rational number given that
the transition probabilities for all outcomes in the set of choices are rational numbers.
\end{thm}

In this paper, based on Theorem \ref{Concept1} and Theorem \ref{Concept2}, we derive explicit formulas and relations for numerous quantities and generating functions stated in Section \ref{section1.1} and Section \ref{section1.2}. We complete this task by using our beloved program, Maple, for symbolic computation. Indeed, we can even go beyond that. It seems that, for each set of choices $R$, we can always find relationships between generating functions for different $n$ (which is one of our main goals in this paper). More interesting patterns and relations can be found by using different sets of choices $R$. It showcases how far one can go with the mindset of an experimental mathematician. 

\begin{rem}
It is worthwhile to mention that all results on generating functions
are obtained in three different ways. The first method is based on counting the paths directly and then \textit{guessing} the rational generating functions from them.
The second method is based on solving the system of linear equations and hence obtaining the rational generating functions.
The third method is based on applying recurrence relations which will be mentioned later. Once the three results are the same,
we are pretty sure that the answers are correct! 
\end{rem}

%%%%%%%%%%%%%%%%%%%%%%%%%%%%%%%%%%%
%%%%%%%%%%%%%%%%%%%%%%%%%%%%%%%%%%%
\section{The case $R = \{1, -1 \}$ for one player } \label{One}

The case when both members of the set of choices $R$ 
are positive have already been done by Wong and Xu \cite{WX} and the second author and Zeilberger \cite{TZ}. In such cases, the games considered are games of pure chance without restricted boundary. 

In this section, we consider the simplest case for the set of choices $R$, i.e., $R = \{1, -1 \}$. We consider this as a pile game with boundary for one player.

\subsection{Probability Generating functions $G_{n,s}(x)$} \label{section2.1}

We recall the following definition of $G_{n,s}(x)$ from Section \ref{section1.1}: 
\[G_{n,s}(x) = \sum_{k=0}^{\infty} B_{n}(k,s)x^k\]
where $B_n(k,s)$ is the probability to end the game
at the $k^{\text{th}}$ turn while the player starts with $s$ chips. 

We obtain the recurrence relation for $G_{n,s}(x)$ in $s$ based on the system of equations
\eqref{Sys1:a} and \eqref{Sys1:b}. We want to solve for the recurrence relation $G_{n,s}(x)$ in parameter $n$. 
We start with the equation \eqref{Sys1:b} for $s=n-1$,
\begin{equation} \label{Ant}
G_{n,n-1}(x) = qx{G_{n,n-2}(x)}+px.  
\end{equation}

Then, we multiply both sides with  
$G_{n-1,s}(x) \cdot G_{n-2,s}(x)$ 
and simplify it by using the combinatorial relation
\begin{equation}  \label{T1}
G_{n,s}(x) = G_{n,m}(x) \cdot G_{m,s}(x), \;\ \;\ 0 \leq s \leq m \leq n,
\end{equation} for the cases $m=n-1$ and $m=n-2$ to get
\begin{equation}  \label{Good}
 G_{n,s}(x) G_{n-2,s}(x) = qxG_{n,s}(x)G_{n-1,s}(x)+pxG_{n-1,s}(x)G_{n-2,s}(x),
\;\  0 \leq s \leq n-2.
\end{equation}

We rewrite this equation in terms of $G_{n,s}(x)$ as
\begin{equation}  \label{Aek2}
  \dfrac{1}{G_{n,s}(x)} 
  = \dfrac{1}{pxG_{n-1,s}(x)}-\dfrac{q}{pG_{n-2,s}(x)},  \;\   0 \leq s \leq n-2. 
\end{equation}

We are able to generate $G_{n,s}(x)$ recursively by
using \eqref{Aek2} when $s=0$ and using \eqref{T1} in the following form:
\begin{equation} \label{T4}
G_{n,s}(x) = \dfrac{G_{n,0}(x)}{G_{s,0}(x)}, \;\ \;\  0 \leq s \leq n
\end{equation}
when $s>0$; along with the trivial initial condition $G_{0,0}(x)=1$. 
For example, we get 
\begin{align*}
G_{0,0}(x) &= 1, \\
G_{1,0}(x) &= \dfrac{px}{1-qx}, \;\ \;\ G_{1,1}(x) = 1,\\
G_{2,0}(x) &= \dfrac{p^2x^2}{1-qx-qpx^2}, \;\ \;\ G_{2,1}(x) = \dfrac{(1-qx)px}{1-qx-qpx^2}, 
\;\ \;\ G_{2,2}(x) = 1.
\end{align*}

We recall that the denominator of $G_{n,s}(x)$ 
is particularly important as it gives the recurrence 
relation for $B_n(k,s)$ in parameter $k$ based on the equation \eqref{Q}.
The fact that the numerator of 
$G_{n,0}(x)$ is $(px)^n$ along with the relation \eqref{Aek2}
give us a way to calculate the denominator recursively. We state it as a corollary.

\begin{cor} \label{Denom}
For any $n$ and $s$, $0\leq s \leq n,$
let $G_{n,s}(x) = \dfrac{P_{n,s}(x)}{Q_n(x)} $. Then
$Q_n(x)$ satisfies the recurrence relation,
\[  Q_n(x) = Q_{n-1}(x)-qpx^2Q_{n-2}(x) ,\] 
with $Q_0(x)=1, Q_1(x)=1-qx.$
\end{cor}

%-----------------------------------------------------------------------------------------------
\subsection{Average Numbers of Turns and Higher Moments}  \label{TwoTwo}
Let $X_{n,s}$ be the random variable of the number of turns required to end the game, i.e., the player collects $n$ chips (or more) 
for the first time starting with $s$. In each turn, the probability of adding one chip to his pile is $p$ and the probability of removing one chip from his pile is $q=1-p$. We consider it as a pile game with boundary.

The first moment is the average number of turns as follows:
\[   E[X_{n,s}] =  \sum_{k=0}^{\infty} k P\{X_{n,s}=k\}.  \]
We derive formulas for $ E[X_{n,s}]$ by using the generating function $G_{n,s}(x)$ as follows:
\[ E[X_{n,s}] = x\dfrac{d\, G_{n,s}(x)}{dx} \bigg\rvert_{x=1}. \]

For the case of $p=q=1/2$, we have a nice polynomial solution for the expectation $E[X_{n,s}]$.
\begin{prop} \label{Prop1} Let $p=1/2$. Then, we have
\[  E[X_{n,s}] = n(n+1)-s(s+1).\]
\end{prop}
\begin{proof}
We only need to verify the solution with the recurrence relations in $s$ and $n$. For the recurrence relation in $s$, we differentiate 
both sides of the equations with respect to $x$ in \eqref{Sys1:a}, \eqref{Sys1:b}, \eqref{Sys1:c}, and then 
multiply both sides by $x$ and set $x=1$  to obtain
\begin{align*}
 E[X_{n,0}] &= 1+\dfrac{1}{2} E[X_{n,1}]+\dfrac{1}{2} E[X_{n,0}],  \\
 E[X_{n,s}] &= 1+\dfrac{1}{2} E[X_{n,s+1}]+\dfrac{1}{2} E[X_{n,s-1}], \;\ \;\ 1 \leq s < n, \\
 E[X_{n,n}] &= 0 . 
\end{align*}for $n\geq1$.

For the recurrence relation in $n$, we apply the same 
differentiation process to equation \eqref{Good} to obtain 

\[   \dfrac{1}{2}E[X_{n,s}] +  \dfrac{1}{2}E[X_{n-2,s}] = 1+ E[X_{n-1,s}].   \]

\end{proof}

%%%%%%%%%%%%%%%%%%%%%%%%%%%
We proceed with the same strategy for the higher moments $E[X_{n,s}^r]$, for $r \geq 1$. 
For each $r$, the formulas for $E[X_{n,s}^r]$ 
are obtained by interpolating
the data generated by the actual generating functions $G_{n,s}(x)$.
For the solutions up to $8^{\text{th}}$-moment, please see the output file.

To verify the formulas rigorously, we apply the operator 
$ \left( \left( x\dfrac{d\,}{dx} \right)^r  
 \bullet \right) \bigg\rvert_{x=1}$ 
to both sides of system of linear equations 
\eqref{Sys1:a}, \eqref{Sys1:b}, \eqref{Sys1:c} and 
\eqref{Good}.
The new conditions are
\begin{align*}
E[X_{n,0}^r] &= E[X_{n,0}^{r-1}]+\dfrac{1}{2} 
\sum_{i=1}^{r}  \binom{r-1}{i-1} (E[X_{n,1}^i]+E[X_{n,0}^i]), \\
E[X_{n,s}^r] &= E[X_{n,s}^{r-1}]+\dfrac{1}{2} 
\sum_{i=1}^{r}  \binom{r-1}{i-1} (E[X_{n,s+1}^i]+E[X_{n,s-1}^i])
,  \;\  \;\   1\leq s<n, \\
E[X_{n,n}^r] &= 0 , \\
\Delta(r,n,n-2) &= \dfrac{1}{2} \sum_{k=0}^r \binom{r}{k}[\Delta(k,n,n-1)+\Delta(k,n-1,n-2)], \\
& \mbox{where} \Delta(r,a,b) = \sum_{i=0}^r \binom{r}{i} E[X_{a,s}^i]E[X_{b,s}^{r-i}].
\end{align*}
We then check the formulas satisfies these new conditions.

\begin{prop}  \label{Prop2}Let $p=q=1/2$. We have the following expressions for the higher moments of the random variable $X_{n,s}$.
 \begin{align*}
&E[X_{n,s}^2] =  \sum_{k=0}^{\infty} k^2 P\{X_{n,s}=k\}
=\dfrac{(n-s)(n+s+1)(5n^2+5n-s^2-s-1)}{3}, \\  
&E[X_{n,s}^3] =  \sum_{k=0}^{\infty} k^3 P\{X_{n,s}=k\}
=\dfrac{(n-s)(n+s+1)}{15}\cdot \\
&(61n^4+122n^3-(14s^2+14s-38)n^2-(14s^2+14s-23)n+s^4+2s^3+8s^2+7s-3).  
\end{align*}
\end{prop}

Higher moments about the mean, i.e., $E[(X_{n,s}-\mu_{n,s})^r]$ for $r>1$, can be calculated once we get the straight moments. We state the results of our computation here.
\begin{prop}  \label{Prop3}Let $p=q=1/2$. We have the following statistics for the random variable $X_{n,s} - \mu_{n,s}$.
\begin{align*}
 &Var[X_{n,s}] = E[(X_{n,s}-\mu_{n,s})^2] 
=\dfrac{(n-s)(n+s+1)(2n^2+2n+2s^2+2s-1)}{3},  \\
 &E[(X_{n,s}-\mu_{n,s})^3] 
=\dfrac{(n-s)(n+s+1)}{15}\cdot \\
&(16n^4+32n^3+8n^2(2s^2+2s+1)+8n(2s^2+2s-1)+16s^4+32s^3+8s^2-8s-3).  \end{align*}
\end{prop}

For the general case of $p$ and $q=1-p$, the expectation  
$E[X_{n,s}^r]$ can be shown to be C-finite sequences in variables $n$ and $s$ respectively (although they are not 
polynomial solutions).
The recurrence in $n$ and $s$ can be seen through the system of equations
\eqref{Sys1:a}, \eqref{Sys1:b}, \eqref{Sys1:c} and the equation \eqref{Good}. It is 
similar to what we have done for the case of $p=q=1/2$ case. For example,
the recurrence relations for $E[X_{n,s}]$ are
\begin{align*}  
(N-1)^2(pN-q) E[X_{n,s}] =0, \\
(S-1)^2(pS-q) E[X_{n,s}] =0
\end{align*}
where $N \cdot E[X_{n,s}] := E[X_{n+1,s}] $
and   $S \cdot E[X_{n,s}] := E[X_{n,s+1}]. $
Based on these recurrence relations, we derive closed-form formulas for $E[X_{n,s}]$ for the general case of $p$ and $q$ algebraically as follows:

\begin{prop}  \label{Prop4}
\[ E[X_{n,s}] = \dfrac{  (2p-1)n+q\left( \dfrac{q}{p}\right)^n 
- \left[(2p-1)s+q\left( \dfrac{q}{p}\right)^s \right] }{(2p-1)^2}   .\]
\end{prop}

Again, it is possible to derive formulas for the general case of $p$ and $q$ for higher moments rigorously. 
We run our program to obtain the annihilators of the equation, i.e., the recurrence relation for $E[X_{n,s}^r]$ for fixed $r$ such that $r>1$ 
from the original data.
Please see the output file for more information about the 
annihilators in $n$ and $s$ for the higher moments.

%-----------------------------------------------------------------------------------------------
\subsection{Formulas for $b_{n}(n+t,s)$}

The theory about pile games with boundary is quite rich in the sense that
there are non-trivial patterns among the assoicated quantities and generating functions everywhere. 
In this section, we formulate the formulas for $b_{n}(k,s)$ 
when $k = n+t$. We recall that $b_n(k,s)$ is the number of ways 
that the game ends (by collecting $n$ chips for the first time)
at the $k^{\text{th}}$ turn while starting with $s$ chips.
Although this is not related to other materials in this paper,
it is an interesting counting problem on its own.

\begin{thm}  \label{Six}
Let $0 \leq s \leq n$ and $-n < t,$ \\ 
Case1: $2| (t+s):$ for $ t-s \leq 2n$,
\begin{equation} \label{c} 
b_n(n+t,s) = \dfrac{n-s}{n+t}\binom{n+t}{\frac{t+s}{2}}.
\end{equation}
Case2: $2 \not| (t+s):$ for $t+s < 2n,$
\begin{equation} \label{d} 
 b_n(n+t,s) = \dfrac{n+s+1}{n+t}\binom{n+t}{\frac{t-s-1}{2}}. 
 \end{equation}
\end{thm}

\begin{proof}
We denote the right hand side of the equation \eqref{c} by $c_n(n+t,s)$ for $2|(t+s)$
and the equation \eqref{d} by $d_n(n+t,s)$ for $2 \not|(t+s).$ 
We first verify the formula for $c_n(n+t,s)$ by 
induction on $t+s=0,2,4,\dots$; and $s=n, n-1, \dots, 0$ 
for each value of $t+s$.

We see that $c_n(n+t,s) = 0$ for $t+s < 0$. 
For the base cases, we have $c_n(n+t,s) = 1$ for $t+s = 0$ 
and $c_n(n+t,n) = 0$ for $-n < t$. 
All of these values trivially agree with the values of $b_n(n+t,s)$.
We will use the lines $t+s=0$ and $s=n$ as boundaries for our induction.
For the induction step, we verify that the formula of 
$c_n(n+t,s)$ satisfies the same
recurrence relation as $b_n(n+t,s)$ in the region $t+s=2,4,6,\dots$ and $s<n$
i.e.
\begin{equation} \label{RecC}
c_n(n+t,s)= c_n(n+t-1,s-1)+c_n(n+t-1,s+1).
\end{equation}
The induction is valid since $(t-1)+(s-1) < t+s$
and $t-1+(s+1) = t+s$ but $s < s+1$. 
We conclude that the equation $c_n(n+t,s)=b_n(n+t,s)$ is true in the region
that both $c_n$ and $b_n$ are defined, i.e.,  $2|(t+s), \;\ 0 \leq s \leq n$ and $t >-n$.

We note that the recurrence relation \eqref{RecC} of
$c_n(n+t,s)$ is also valid for $s \leq 0$. 
For example, if $s=0$, we have
\begin{align}    c_n(n+t,0) &= c_n(n+t-1,-1)+c_n(n+t-1,1).    \label{b4}       \end{align}
By the equation (\ref{b1}), we have the following relation
\begin{align} b_n(n+t,0) &= b_n(n+t-1,0)+b_n(n+t-1,1).\label{b5}\end{align}
By comparing the recurrence relations (\ref{b4}) and (\ref{b5}), we define $d_n(n+t-1,0)= c_n(n+t-1,-1)$.
By doing this comparison inductively for $s<0$, 
we have the following relation: \[  d_n(n+t,-s-1)= c_n(n+t,s).  \]

Lastly, the induction terminates on another side of the region. At the point $(t,s)=(n+1,-n-1)$, we see that 
$c_n(2n+1,-n-1)=d_n(2n+1,n)=1$ while 
$b_n(2n+1,n)=0.$ Hence, all the induction steps are terminated from this point and no more induction step is required. Therefore, another restriction for the 
valid region of $c_n$ is $t-s \leq 2n.$
\end{proof}

We note that we prove the formula in Theorem \ref{Six} algebraically. An alternative proof by combinatorial methods for this result is very welcome.

%%%%%%%%%%%%%%%%%%%%%%%%%%%%%%%%%%%%%%
%%%%%%%%%%%%%%%%%%%%%%%%%%%%%%%%%%%%%%
%%%%%%%%%%%%%%%%%%%%%%%%%%%%%%%%%%%%%%
\section{The case $R=\{1,-1\}$ for two players } \label{Two}

In this section, we consider the set of choices $R=\{1,-1\}$. We consider this as a pile game with boundary for two players. We compute the winning probability of the first player. We also compute other statistics of
the number of rounds required until the first player wins or until either of the player wins
the game.

\subsection{The winning Probability Generating Function of the First Player}

We recall from Section \ref{section1.2} that the winning probability generating functions 
of the first player is defined by
\[W_{n,s_1,s_2}(x) := \sum_{k=0}^{\infty} w_{n}(k,s_1,s_2)x^k\]
where $w_{n}(k,s_1,s_2)$ is the probability
that the first player wins (by being the first one to collect $n$ chips or more) 
at the $k^{\text{th}}$ turn; and the first and the second players start with $s_1$ and  $s_2$ 
chips respectively. 

Let $C_n(k,s)$ be the probability that none of the players win the game up to their $k^{\text{th}}$ turn while starting with 
$s$ chips. We note that 
\[ C_n(k,s) = 1- \sum_{i=0}^k B_n(i,s).  \]
The probability generating function of $C_n(k,s)$ is defined by
\[ H_{n,s}(x) :=  \sum_{k=0}^{\infty} C_{n}(k,s)x^k.   \] 
This function satisfies the following relation:   
\[ H_{n,s}(x) = \dfrac{1-G_{n,s}(x)}{1-x}. \]

For a fixed value of $k$, we note that \[w_{n}(k,s_1,s_2) = B_{n}(k,s_1) \cdot C_n(k-1,s_2).\]
The term $w_n(k,s_1,s_2)$ is then the product of the diagonal terms of the generating functions 
$G_{n,s_1}(x)$ and $xH_{n,s_2}(x)$. 

The first method to find the formulas for $W_{n,s_1,s_2}(x)$ 
is by the residue theorem. More precisely, we find the residue
of \[ \dfrac{G_{n,s_1}(x/t)\cdot tH_{n,s_2}(t)}{t} \] 
at $t=m_i$ where $m_i$ are the singularities of this 
function and satisfies $\lim_{x \rightarrow 0}m_i(x) =0$.
Although this method is sound, the residues are very difficult to compute even with the use of computer. Alternatively, 
the second method by fitting the data into certain mathematical expression is much simpler 
and more practical. By Theorem 4.2(2) in the texbook written by Kauers and Paule \cite[p.71]{KP},
we know that the sequence $w_{n}(k,s_1,s_2)$ in $k$, as a point-wise product 
of two C-finite sequences, is a C-finite sequence. The order of the recurrence relation is bounded by the product of the order of $B_n$ and the order of $C_n$. 
Hence, by Theorem \ref{Concept2}, 
we interpolate the rational generating function 
$W_{n,s_1,s_2}(x)$ such that the degree of the 
denominator is bounded by this number. We state it as a theorem.

\begin{thm}
For any $n$, the winning probability generating functions $W_{n,s_1,s_2}(x)$ of the pile game with boundary for two players are all rational functions
with the same denominator of degree at most $n(n+1)$.  
Equivalently, the terms $w_n(k,s_1,s_2)$ as a sequence in $k$ 
form a C-finite sequence of order at most $n(n+1)$;  
and the recurrence relations are the same for each value of $s_1$ and $s_2$ 
(with different initial values, of course).
\end{thm}

\begin{proof}
By Theorem \ref{Concept1}, the order of $B_{n}(k,s_1)$ 
is at most $n$ for $0 \leq s_1 < n$. Also, by Theorem 4.2(3) in the textbook written by Kauers and Paule \cite[p. 71]{KP}, 
$C_n(k,s) = 1- \sum_{i=0}^k B_n(i,s)$ is of order at most $n+1$.
The order of $w_{n}(k,s_1,s_2) = B_{n}(k,s_1) \cdot C_n(k-1,s_2)$ 
is therefore bounded by $n(n+1)$ by the same theorem.
The statement about the generating function $W_{n, s_1, s_2}(x)$ is true by the equation \eqref{Q}.
\end{proof}

For example, let $p=q=1/2$. The rational functions $W_{n,s_1,s_2}(x)$ for some $n$, $s_1$ and $s_2$ are listed as follows:
\begin{align*}
  W_{1,0,0}(x) &= \dfrac{2x}{4-x}, \\   
  W_{2,0,0}(x) &= \dfrac{2x^2(8-x)}{(x+4)(x^2-12x+16)},\\
  W_{3,0,0}(x) &= \dfrac{ -16x^3(x^2+10x-32)}{(x^3-24x^2+80x-64)(x^3+4x^2-32x-64)}.      
\end{align*}

Please see the output file for the rational functions $W_{n,0,0}(x)$, for $1\leq n \leq 12$. 
Unlike Corollary \ref{Denom} in Section \ref{One}, 
the denominators of  $W_{n,0,0}(x)$ do not seem to have any nice relation in $n$ (but this may be too much to ask for).
Also, we look at the \textit{less complex} sequence of the winning probability of the 
first player, $\bar{w}(n) := W_{n,0,0}(1)$ for $1\leq n \leq 15$. We list their values as follows:
\[ \dfrac{2}{3}, \dfrac{14}{25}, \dfrac{48}{91}, \dfrac{8752}{16983}, \dfrac{54402}{106711}, 
\dfrac{31234254}{61625005}, \dfrac{13536472628672}{26801524306425}, 
\dfrac{22605208036288}{44860083149401}, \dfrac{604879643045690}{1202291036053231}, \]
\[ \dfrac{20845854034408931883490}{41481744738197592962679}, 
\dfrac{10267708004766652604842480}{20449419753638432679768691}, 
\dfrac{140038998291791920344242305648}{279087176779522234199477629375},\] \[\dfrac{161633807302263156492382071431905658}{322288621844885789730727320061228179}, \dfrac{2354603411248560275753127681958620368325422}{4696859311668431636549510834091264555161333},\] \[ \dfrac{389133153474651298255081635215667799647913216}{776482226164826584565197354645471592297050103}.    \]

This sequence does not seem to satisfy any linear recurrence relation.
If this observation were true, it would be the opposite circumstance
to the other version of pile games in which the player
can collect any negative number of chips, as discussed in the papers \cite{TZ, LT}. In such cases, the sequence $\bar{w}(n)$  is \textit{holonomic} of small order (see 
eq.(10) \cite[p.9]{LT} 
and Proposition 7 in \cite[p.19]{TZ}). We post this observation as a conjecture.
 
\begin{conj}
The winning probability of the first player 
$\bar{w}(n)$ is not a holonomic sequence,
i.e., it cannot be written as a linear 
recurrence relation with polynomial coefficients as follows:
\[ p_0(n)\bar{w}(n)+p_1(n)\bar{w}(n-1)+\dots+p_N(n)\bar{w}(n-N)=0\]
for some specific $N$ and some polynomials $p_0(n), p_1(n), \dots, p_N(n)$.
\end{conj}

%---------------------------------------------------------------------------------------------
\subsection{Average Number of Turns and Higher Moments}
Some readers may notice that we did not mention our favorite trick
that was used in the papers \cite{LT, TZ, WX},
\[ \bar{w}(n) = \dfrac{1}{2}+\dfrac{1}{2}\sum_{k=0}^{\infty} B_n(k,0)^2.\]
Instead, we choose to find the generating functions. 
The advantage of the generating-function approach is that the generating function contains more information for deriving other statistics related to the problem. 
We list statistics of some random variables in this section. In general, we think that (as a consequence of the values of $\bar{w}(n)$)
there is no more interesting pattern to talk about. 

Let $Y_{n}$ be the random variable of the number of rounds required until
the first player wins the game by collecting $n$ chips first.
Both players start with no chip. In each turn, the probabilities of adding one chip to the pile or removing one chip from the pile are both $1/2$.
If the player has nothing in his pile and chooses to remove one chip in the next turn, then the number of chips in his pile will still be zero. 

For $n=1$,  the sequence of the straight moment, $E[Y_1^r], \;\ r=0,1,2,3,\dots$ is
\[  1, \dfrac{4}{3}, \dfrac{20}{9}, \dfrac{44}{9}, \dfrac{380}{27}, 
\dfrac{4108}{81}, \dfrac{17780}{81}, \dfrac{269348}{243}, 
\dfrac{4663060}{729}, \dfrac{10091044}{243}, 
\dfrac{218374420}{729}, \dots   \]

The sequence of the moment about the mean, $E[(Y_1-E[Y_1])^r], \;\ r=1,2,3,\dots$ is
\[0, \dfrac{4}{9}, \dfrac{20}{27}, \dfrac{20}{9}, \dfrac{1940}{243}, 
\dfrac{25204}{729}, \dfrac{14140}{81}, \dfrac{6609460}{6561}, 
\dfrac{128728340}{19683}, \dfrac{103175452}{2187}, \dots\]

On the other hand, let $Z_n$ be the random variable of the number of turns required to end the game, i.e., either the first player wins (by collecting $n$ chips first) or the second player wins. 
We note that each player's play is counted as one turn; and one round has two turns. 

For $n=1$,  the sequence of the straight moment, $E[Z_1^r], \;\ r=0,1,2,3,\dots$ is
\[ 1, 2, 6, 26, 150, 1082, 9366, 94586, 1091670, 14174522, 204495126, \dots   \]

The sequence of the moment about the mean, $E[(Z_1-E[Z_1])^r], \;\ r=1,2,3,\dots$ is
\[0, 2, 6, 38, 270, 2342, 23646, 272918, 3543630, 51123782, \dots\]

Please see the output files for the moments of $Y_n$ and $Z_n$ for other values of $n$. 

%%%%%%%%%%%%%%%%%%%%%%%%%%%%%%%%%%%%%
%%%%%%%%%%%%%%%%%%%%%%%%%%%%%%%%%%%%%

\section{The case $R = \{1, -u \}$ for one player, where $ u=1,2,3,\dots$} \label{section4}

In this section, we look at the set of choices $R=\{1,-u\}$ such that the probability of adding one chip to the pile is $p$; and the probability of removing $u$ chips, $u\geq 1$, is $q=1-p$. We consider it as a pile game with boundary.

\subsection{Probability Generating functions, $G_{n,s}(x)$}

We recall the definition of the generating function $G_{n,s}(x)$ first:
\[G_{n,s}(x) := \sum_{k=0}^{\infty} B_{n}(k,s)x^k\]
where $B_n(k,s)$ is the probability that the game ends
(by collecting $n$ chips for the first time)
at the $k^{\text{th}}$ turn while starting with $s$ chips. 

For the recurrence relation of $G_{n,s}(x)$ as a sequence in $n$,
we start with the relation
\[G_{n,n-1}(x) = qx{G_{n,n-u-1}(x)}+px.    \]
Then, we combine this relation with the same combinatorial 
relation that was mentioned in Section \ref{section2.1},
\begin{equation}  \label{Rela3}
G_{n,s}(x) = G_{n,m}(x) \cdot G_{m,s}(x), \;\ \;\  0 \leq s \leq m \leq n. 
\end{equation}
By the same approach as in Section \ref{section2.1}, the functions $G_{n,s}(x)$ satisfy another the relation in $n$ as follows:
\begin{equation}  \label{Rela4}
  \dfrac{1}{G_{n,s}(x)} = \dfrac{1}{pxG_{n-1,s}(x)}-\dfrac{q}{pG_{n-u-1,s}(x)}. 
\end{equation}

Then again, for any $n$ and $s, \;\ 0 \leq s \leq n,$
we generate $G_{n,s}(x)$ by applying the
relation \eqref{Rela4} when $s=0$  and the relation \eqref{Rela3} in the following form: 
$G_{n,s}(x) = \dfrac{G_{n,0}(x)}{ G_{s,0}(x)}$ when $s>0$;
along with trivial initial value $G_{0,0}(x)=1.$ \\

It is important to stress that the denominators of the functions 
$G_{n,s}(x)$ are similar to those in Corollary \ref{Denom}. We state it as a corollary here.

\begin{cor} 
For any $n$ and $s,$ $0\leq s \leq n$ and the choice set 
$R = \{1,-u\}, u \geq 1,$
let $G_{n,s}(x) = \dfrac{P_{n,s}(x)}{Q_n(x)} $. Then
$Q_n(x)$ satisfies the recurrence relation,
\[  Q_n(x) = Q_{n-1}(x)-qp^ux^{u+1}Q_{n-u-1}(x).\] 
\end{cor}

%---------------------------------------------------------------------------------------------
\subsection{Average Numbers of Turns and Higher Moments} 

Similar to Section \ref{TwoTwo}, let $X_{n,s}$ be the random variable of the number of turns required to end the game, i.e., the player collects $n$ chips or more 
for the first time while starting with $s$ chips.  

We start with the expected number of this random
variable for different values of $n$ and $s$. It is clear that the bigger the value of $u$, 
the higher the number of $E[X_{n,s}]$.
We first consider the case where $p=1/2$ and $q=1/2$. This case provides nice solutions to various statistics we considered. The solutions
are collected in the proposition below. 

To verify that the following proposition is correct, 
we only need to check the following recurrence relations
in $n$ and $s$ which can be obtained in essentially the same way 
to those derived in the proof of Proposition \ref{Prop1}. For $n \geq 1$, we have 
\begin{align*}
 E[X_{n,s}] &= 1+\dfrac{1}{2} E[X_{n,s+1}]+\dfrac{1}{2} E[X_{n,0}], \;\ \;\ 0 \leq s < u,  \\
 E[X_{n,s}] &= 1+\dfrac{1}{2} E[X_{n,s+1}]+\dfrac{1}{2} E[X_{n,s-u}], \;\ \;\ u \leq s < n, \\
 E[X_{n,n}] &= 0 , \\
  \dfrac{1}{2}E[X_{n,s}] +  \dfrac{1}{2}E[X_{n-u-1,s}] &= 1+ E[X_{n-1,s}].   
\end{align*}

\begin{prop}Let $p=q=1/2$. We have the following formula for $E[X_{n,s}]$:
\[  E[X_{n,s}] = 2[C_u(n)-C_u(s)], \;\ \;\ 0\leq s \leq n, \]
where
\begin{align*}     C_u(n) &=  2C_u(n-1)+1-C_u(n-u-1) \text{ and }
C_u(0)=C_u(-1)=\dots=C_u(-u)=0.   \end{align*}
\end{prop}

For the general case of $p$ and $q=1-p$,
we have the following recurrence relations of $E[X_{n,s}]$:
\begin{align}  
(S-1)^2(pS^u-q(S^{u-1}+S^{u-2}+\dots+1))E[X_{n,s}] &= 0,  \label{Pattern:a} \\
(N-1)^2(pN^u-q(N^{u-1}+N^{u-2}+\dots+1))E[X_{n,s}] &= 0. \label{Pattern:b}
\end{align}

For example, if $u=2$, $p=2/3$ and $q=1/3$, 
we have the following relations.
\[   (S-1)^3(2S+1)E[X_{n,s}]   = 0,\]
\[  (N-1)^3(2N+1)E[X_{n,s}]   = 0.\]
Based on these relations, we derive the closed-form formula for $E[X_{n,s}]$ for the general case of $p$ and $q$ algebraically as follows:

\begin{prop} \label{prop12}
Let $u=2$, $p=2/3$ and $q=1/3$. We have the following formula for $E[X_{n,s}]$:
\[  E[X_{n,s}] = \dfrac{n(3n+5)}{6} - \dfrac{1}{9}\left(\frac{-1}{2}\right)^n
-\left[\dfrac{s(3s+5)}{6}-\dfrac{1}{9}\left(\frac{-1}{2}\right)^s \right].\]
\end{prop}

Higher moments can be obtained automatically for each $u$. 
For example, if we have $R=\{1,-2\}$, the terms $E[X_{n,s}^2]$ 
satisfy the recurrence relation
\begin{align*}
(N-1)^3(pN+q)(p^2N^2-(p+1)qN+q^2)(pN^2-qN-q)^2 \cdot  E[X_{n,s}^2] &= 0, \\
(S-1)^3(pS^2-qS-q)^2 \cdot  E[X_{n,s}^2] &= 0.
\end{align*}

But, to find the grand pattern for general case of any value of $u$ as in 
\eqref{Pattern:a} and \eqref{Pattern:b} for the higher moments, 
we still need to finish this task by hands. 
Please see the output file for different values of $u$ and various higher moments in this case.

As discussed in Section \ref{Two}, the winning probability
for the first player $\bar{w}(n)$ does not seem to have a
pattern for the case $\{1,-1\}$. It is likely that the winning probability $\bar{w}(n)$ for the first player does not have any pattern for the general case of   
$\{1,-u\}$ also. 
Although the curious reader are welcome to explore this problem by our 
program with the function 
\texttt{ProbWin(n,R)} where $R= [1,-u]$
for some $u=1,2,3,\dots$.

%%%%%%%%%%%%%%%%%%%%%%%%%%%%%%%%%%

\section{The case $R = \{2, -1 \}$ for one player} \label{section5}

In this section, we consider the set of choices $R=\{2,-1\}$ such that the probability of adding two chip to the pile is $p$; and the probability of removing one chip from the pile is $q=1-p$. We consider it as a pile game with boundary.

We first find some relations for the probability generating function in the single-player scenario. We have 
\[G_{n,s}(x) = \sum_{k=0}^{\infty} B_{n}(k,s)x^k\]
where $B_n(k,s)$ is the probability to end the game
at the $k^{\text{th}}$ turn while starting with $s$ chips. 

We expected that we get similar results to those in the
case of $R=\{1,-u\}$ by using essentially the same approach as in Section \ref{section4}. But it turns out that this case is
more complicated.

There are two possibilities to end the game: The player adds two chips to his pile of $n-2$ chips; or the player adds one chip to his pile of $n-1$ chips. 
Let the generating function for the first and second case 
be $P_{n,s}(x)$ and $Q_{n,s}(x)$ respectively. It is clear that 
\[  G_{n,s}(x)=P_{n,s}(x)+Q_{n,s}(x),  \;\  0 \leq s \leq n. \]

We find the expressions for $P_{n,s}(x)$ and $Q_{n,s}(x)$ inductively and hence obtain the expression for 
$G_{n,s}(x)$. 
 
By definition, we have $P_{n,n}(x)=1$, $Q_{n,n+1}(x)=1$,  
$P_{n,n+1}(x)=0$ and $Q_{n,n}(x)=0$. 
Then, by the recurrence relation similar to the relation \eqref{Ant}, for $n \geq 2$, we have
\begin{align*}
P_{n,n-1}(x) &= qxP_{n,n-2}(x) ,     \\
Q_{n,n-1}(x) &= qxQ_{n,n-2}(x)+px .   
\end{align*}
 
Based on the combinatorial relations similar to the relation (\ref{T1}), we also have  
\begin{align*}
P_{n,s}(x) &= P_{n,m}(x) \cdot P_{m,s}(x) +  P_{n,m+1}(x) \cdot Q_{m,s}(x), \;\ \;\  0 \leq s \leq m \leq n,\\
Q_{n,s}(x) &= Q_{n,m}(x) \cdot P_{m,s}(x) +  Q_{n,m+1}(x) \cdot Q_{m,s}(x), \;\ \;\  0 \leq s \leq m \leq n.
\end{align*} 
 
By setting $m=n-1,$ we obtain
\begin{align}
 P_{n,s}(x) &= P_{n,n-1}(x) \cdot P_{n-1,s}(x)+Q_{n-1,s}(x),    \;\ 0 \leq s \leq n-1,  \label{A1} \\
 Q_{n,s}(x) &= Q_{n,n-1}(x) \cdot P_{n-1,s}(x),  \;\ 0 \leq s \leq n-1.  \label{A2} 
\end{align} 
 
Next, we combine these recurrence relations 
with the combinatorial relations by setting $m=n-1$ and $s=n-2$. We obtain the following expressions:
\begin{align}
 Q_{n,n-1}(x) &= \dfrac{px}{1-qxP_{n-1,n-2}(x)},  \label{End1} \\
 P_{n,n-1}(x) &= \dfrac{qx \cdot Q_{n-1,n-2}(x)}{1-qxP_{n-1,n-2}(x)}
  = \dfrac{q}{p}Q_{n,n-1}(x) \cdot Q_{n-1,n-2}(x).  \label{End2}
\end{align}

With these four relations \eqref{A1},  \eqref{A2}, \eqref{End1}
and \eqref{End2}
 (the first two for $s<n-1$ and the last two for $s=n-1$)
along with initial conditions $P_{1,0}(x)=0$ and $Q_{1,0}(x)=\dfrac{px}{1-qx}$, 
we are able to generate the expressions for $P_{n,s}(x)$ and $Q_{n,s}(x)$ and hence the expressions for $G_{n,s}(x)$ recursively. 
Some examples are listed below:
\begin{align*}
G_{0,0}(x) &= 1, \\
G_{1,0}(x) &= \dfrac{px}{1-qx}, \;\ \;\ G_{1,1}(x) = 1,\\
G_{2,0}(x) &=  \dfrac{px}{1-qx}, \;\ \;\ G_{2,1}(x) =  \dfrac{px}{1-qx},
\;\ \;\ G_{2,2}(x) = 1, \\
G_{3,0}(x) &=  \dfrac{p^2x^2(1+qx)}{1-qx-pq^2x^3}, \;\ \;\ 
G_{3,1}(x) =  \dfrac{px(1-qx+pqx^2)}{1-qx-pq^2x^3}, \;\ \;\ 
G_{3,2}(x) = \dfrac{px(1-qx)(1+qx)}{1-qx-pq^2x^3}, \\
G_{3,3}(x) &= 1.
\end{align*}

We have the following statement about the denominator of $G_{n,s}(x)$.
\begin{cor} 
For any $n$ and $s$ such that $0\leq s \leq n$, given the set of choices $R = \{2, -1\}$,
let $G_{n,s}(x) = \dfrac{N_{n,s}(x)}{D_n(x)}$ such that the constant term of $D_n(x)$ is $1$. 
Then, the terms $D_n(x)$, as a sequence in $n$, satisfy the following recurrence relation:
\[  D_n(x) = D_{n-1}(x)-pq^2x^{3}D_{n-3}(x).\] 
\end{cor}

\begin{proof}
By solving the system of linear equations similar to the one in 
\eqref{Sys1:a}, \eqref{Sys1:b} and \eqref{Sys1:c}, it is true that the denominators of $Q_{n,s}(x)$, $P_{n,s}(x)$
and $G_{n,s}(x)$ are the same for each $n$. 

Let $Q_{n,s}(x) = \dfrac{A_{n,s}(x)}{D_n(x)}$
and $P_{n,s}(x) = \dfrac{B_{n,s}(x)}{D_n(x)}$. \\

We write the relation \eqref{End2} as
\[   \dfrac{A_{n,n-1}(x)}{D_n(x)} = \dfrac{pxD_{n-1}(x)}{D_{n-1}(x)-qxB_{n-1,n-2}(x)}. \]
Hence, we get the relations
\begin{align}
A_{n,n-1}(x) &= C(x)pxD_{n-1}(x), \label{End3} \\
D_n(x) &= C(x)[ D_{n-1}(x)-qxB_{n-1,n-2}(x) ], \label{End4}
\end{align}
for some function $C(x)$ which does not depend on $n$. 

We simplify the relation \eqref{End1} by the relation \eqref{End3}, we get
\[ B_{n,n-1}(x) = \dfrac{q}{p}\dfrac{A_{n,n-1}(x)}{D_{n-1}(x)}A_{n-1,n-2}(x) 
= C(x) \cdot pqx^2D_{n-2}(x).  \]

Finally, by substituting this identity to the relation \eqref{End4}, we get 
the following relations among the terms $D_{n}$:
\[  D_n(x) = C(x)[ D_{n-1}(x)-C(x)pq^2x^3D_{n-3}(x) ].\] 

We solve for $C(x)$ by using some fixed values of $n$. For example, we get $C(x)=1$ by setting $n=4$.
\end{proof}

Now we turns our attention to higher moments of $G_{n,s}(x)$. 
It is a little surprise that we cannot find any pattern not only for $G_{n,s}(x)$ but also for  
$P_{n,s}(x)$  and $Q_{n,s}(x)$  as well.
One may expect that we would get some nice results similar to those in Proposition \ref{Prop1},
Proposition \ref{Prop2}, Proposition \ref{Prop3} and Proposition \ref{Prop4} in Section \ref{TwoTwo}.
But it is also reasonable that it shows no pattern. 
At the end, the relations like
\eqref{A1}, \eqref{A2}, \eqref{End1}, \eqref{End2} are
much more complicated than, say, the relations \eqref{Aek2} and \eqref{Rela4}. 
We list some of these numbers in the output file for the readers who are interested in the computation we have done.

%%%%%%%%%%%%%%%%%%%%%%%%%%%%%%%%%%%
%%%%%%%%%%%%%%%%%%%%%%%%%%%%%%%%%%%
\section{Conclusion}
We consider pile games with boundary. We consider the probability generating functions of the number of turns required to end the game in both one-player scenario and two-player scenario for various set of choices $R$. For 
each $R$ considered, we obtain the probability generating function by solving 
the linear system of recurrence relations for the quantities involved. This process guarantees that the functions $G_{n,s}(x)$ 
and $W_{n,s_1,s_2}(x)$ are rational functions. 

For the set of choices $R=\{1,-u\}, \,u=1,2,\dots$, we manage to define 
the probability generating function $G_{n,s}(x)$ for one player recursively. As an application, other statistics can be calculated rapidly. In the two-player scenario, we conjecture that the winning probability $\bar{w}(n)$ for 
the first player shows no pattern, i.e., $\bar{w}(n)$ is not holonomic.
 
We are also interested in the case $R=\{u,-1\}, \,u=1,2,\dots$.
We manage to define the generating functions 
$G_{n,s}(x)$ recursively but we only did it explicitly for the case 
$\{2,-1\}$. It is feasible to obtain analogous results for other cases but we left it open for the readers as we are exhausted after all the computation we have done so far.

It is quite likely that many other formulas, recurrence relations and patterns can be obtained in ways similar to those used above for other pile games (with or without boundary). It is not our purpose to provide an exhaustive list, but to illustrate the method and in particular the usefulness of our approach. 

We would like to thank the computer program Maple that gave us most of
the patterns and formulas mentioned in this paper. It would be impossible to finish this project without the help
of our computer friend.

%%%%%%%%%%%%%%%%%%%%%%%%%%%%%%%%%%%
%%%%%%%%%%%%%%%%%%%%%%%%%%%%%%%%%%%

\end{document}